\documentclass{amsart}
 \newtheorem{thm}{Theorem}[section]
 \newtheorem{cor}[thm]{Corollary}
 \newtheorem{lem}[thm]{Lemma}
 \newtheorem{prop}[thm]{Proposition}
 \theoremstyle{definition}
 
 \theoremstyle{remark}
 \newtheorem{rem}[thm]{Remark}
 
 \numberwithin{equation}{section}
\DeclareMathOperator{\Hom}{Hom} \DeclareMathOperator{\Ext}{Ext}
\DeclareMathOperator{\Spec}{Spec}

\DeclareMathOperator{\depth}{depth} \DeclareMathOperator{\Ht}{ht}

\newcommand{\fm}{\mathfrak{m}}

\newcommand{\fp}{\frak{p}}
\newcommand{\fq}{\frak{q}}

\newcommand{\fn}{\frak{n}}

\begin{document}
%
%
%
%
%
%
%
%
%
\title[A NOTE ON QUASI-GORENSTEIN RINGS]
 {A NOTE ON QUASI-GORENSTEIN RINGS}
\author[S. H. Hassanzadeh]{S. H. Hassanzadeh}

\address{Faculty of Mathematical Sciences and Computer, Tarbiat Moallem University, 599 Taleghani Ave., Tehran 15618, Iran}
\address{Institut de mathematiques, Universite Paris 6, 4, Place
Jussieu, F-75252 Paris Cedex 05, France}

\email{h\_hassanzadeh@tmu.ac.ir}
\author[N. Shirmohammadi]{N. Shirmohammadi}

\address{Faculty of Mathematical Sciences and Computer, Tarbiat Moallem University, 599 Taleghani Ave., Tehran 15618, Iran}

\email{ne.shirmohammadi@gmail.com}
\author[H. Zakeri]{H. Zakeri}

\address{Faculty of Mathematical Sciences and Computer, Tarbiat Moallem University, 599 Taleghani Ave., Tehran 15618, Iran}

\email{zakeri@saba.tmu.ac.ir}



\date{January 1, 2004}

\begin{abstract}
In this paper, after giving a criterion for a Noetherian local
ring to be quasi-Gorenstein, we obtain some sufficient conditions
for a quasi-Gorenstein ring to be Gorenstein. In the course, we
provide a slight generalization of a theorem of Evans and
Griffith.
\end{abstract}

\maketitle
\section{Introduction}
Let $R$ be a Noetherian local ring with maximal ideal $\fm$ and
 residue field $k$. Throughout,  we use $E_{R}(k)$ to denote the
 injective hull of $k$. Also, for a non-negative integer $i$, we
 use $H_{\fm}^{i}(R)$ to denote the $i$-th local cohomology module
 of $R$ with respect to $\fm$. Following \cite{AG}, we say that $R$ is a
quasi-Gorenstein ring if $H_{\fm}^{dim R}(R)\cong E_{R}(k)$. In
this paper, we investigate quasi-Gorenstein rings in terms of
linkage. By the concept of linkage, we mean the complete
intersection linkage. More precisely, for ideals $I$ and $J$, we
say that $I$ is linked to $J$ (written $I\sim J$) if there is an
$R$-sequence $x_{1},\ldots,x_{n}$ in $I\cap J$ such that
$I=(x_{1},\ldots,x_{n}):J$ and $J=(x_{1},\ldots,x_{n}):I$.

To begin, using Cohen's structure theorem in conjunction with the
linkage theory, we provide a criterion which shows that the ring
$R$ is quasi-Gorenstein if and only if $\widehat{R}$, the
$\fm$-adic completion of $R$, is a certain specialization of a
Gorenstein local ring. This enables us to compare (in 2.4) the
Gorenstein loci and the Cohen--Macaulay loci of a
quasi-Gorenstein ring. The question asking when a
quasi-Gorenstein ring is Gorenstein constitutes a lot of
interests. In this direction, we use the above mentioned
criterion to establish the results 2.5 and 2.7 which provide some
sufficient conditions under which a quasi-Gorenstein ring is
Gorenstein. Finally, using 2.7 and a theorem of Serre, we provide
a slight generalization of a theorem of Evans and Griffith
concerning the problem to know when the residue class ring of a
height two prime ideal of a certain regular local ring is
complete intersection.


\section{The Results}

It is a basic fact that quasi-Gorenstein rings satisfy Serre's
condition $S_{2}$. Although this fact has already been proved,
but, for the reader's convenience, in the following lemma,  we
provide a different proof by using local cohomology tools.
Following \cite{M}, we say that a proper ideal $I$ in a
Noetherian ring is unmixed if the heights of its prime divisors
are all equal.

\begin{lem}
Every quasi-Gorenstein ring satisfies $S_{2}$. In particular, its
zero ideal is unmixed.
\end{lem}
\begin{proof}
Let $(R,\fm,k)$ be a quasi-Gorenstein ring of dimension $n$. To
prove that $R$ satisfies $S_{2}$, without loss of generality, we
can assume that $R$ is complete (cf. \cite[2.1.16]{BH}); so that,
by Cohen's structure theorem, there exists a regular local ring
$S$ of dimension $n'$ such that $R\cong S/I$ for some ideal $I$ of
$S$. Now, by Local Duality Theorem (cf. \cite[11.2.6]{BSh}) and
\cite[1.2.4]{BH}, $R\cong \Hom_{R}(E_{R}(k),E_{R}(k))\cong
\Hom_{R}(H_{\fm}^{n}(R),E_{R}(k))\cong\Ext^{n'-n}_{S}(R,S)\cong
\Hom_{S/(\textit{\textbf{x}})}(R,S/(\textit{\textbf{x}}))$, where
$\textit{\textbf{x}}$ is a maximal $S$-sequence in $I$. Using
this in conjunction with \cite[1.4.19]{BH},  one can deduce that
$R$ satisfies $S_{2}$. Now, for a prime divisor $\fp$ of $R$, we
have $0=\depth R_{\fp} \geq \min (2,\Ht\fp)$; hence the final
statement holds.
\end{proof}

In \cite[5.2]{BE}, Buchsbaum and Eisenbud established relations,
under certain conditions, between almost complete intersection
ideals and Gorenstein ideals. Then, Schenzel, in \cite{Sch1},
improved the above result by establishing a duality between
almost complete intersection ideals and quasi-Gorenstein ideals.
The next  proposition, which is related to the above mentioned
result, shows that one can characterize quasi-Gorenstein rings by
using the concept of linkage.

In the proof of the next proposition, for a local ring $(R,\fm)$,
we use $\widehat{R}$ to denote the $\fm$-adic completion of $R$.

\begin{prop}
Let $S$ be a Gorenstein local ring and $I$ be an ideal of $S$ of
height zero. Then $S/I$ is quasi-Gorenstein if and only if $I\sim
Sx$ for some $x\in S$.
\end{prop}
\begin{proof}
$(\Leftarrow)$ Let $R=S/I$ and assume that $\fm$ is the maximal
ideal of $R$ and that $n=\dim R$. Then $R=S/0:_{S}x\cong
Sx=0:_{S}I\cong \Hom_{S}(R,S)$; so that, by Local Duality Theorem,
$E_{R}(R/\fm)=\Hom_{R}(R,E_{R}(R/\fm))\cong
\Hom_{R}(\Hom_{S}(R,S),E_{R}(R/\fm))\cong H_{\fm}^{n}(R)$.
Therefore $R$ is quasi-Gorenstein. \\
$(\Rightarrow)$ Since, by 2.1, $I$ is unmixed, it follows from
\cite[2.2]{Sch2} that $I\sim(0:_{S}I)$. Therefore we have only to
prove that $(0:_{S}I)$ is a principal ideal. Using the
hypothesis, we easily see that $\widehat{S}/I\widehat{S}$ is also
a quasi-Gorenstein ring. Thus, applying $- \otimes_S \widehat{S}$
to the relation $I\sim(0:_{S}I)$, one obtains the relation
$I\widehat{S}\sim(0:_{\widehat{S}}I\widehat{S})$. Therefore, by
the same arguments as in the proof of 2.1, we have
$\widehat{S}/I\widehat{S}\cong
\Hom_{\widehat{S}}(\widehat{S}/I\widehat{S},\widehat{S})\cong
0:_{\widehat{S}}I\widehat{S}=(0:_{S}I)\widehat{S}.$ Hence, by
\cite[Exercise 8.3]{M}, $(0:_{S}I)$ is a principal ideal.
\end{proof}

\begin{rem}
It is straightforward to see that a Noetherian local ring
$(R,\fm)$ is quasi-Gorenstein if and only if $\widehat{R}$ is so.
On the other hand, according to Cohen's structure theorem,
$\widehat{R}$ can be described as a residue class ring of a
Gorenstein local ring with the same dimension. Therefore, 2.2
provides a criterion to justify whether a given Noetherian local
ring is quasi-Gorenstein.
\end{rem}

The next corollary shows that, for a quasi-Gorenstein ring $R$,
the Gorenstein loci ($Gor(R)$) and the Cohen--Macaulay loci
($CM(R)$) coincide.

In the course of this paper, for a finitely generated $R$-module
$M$, $r(M)$ and $\mu (M)$ will denote the type of $M$ and the
minimal number of generators of $M$, respectively.

\begin{cor}
$Gor(R)=CM(R)$ whenever $R$ is a quasi-Gorenstein ring.
\end{cor}
\begin{proof}
For a prime ideal $\fp$ of $R$, one can use \cite[Theorem 23.2(i)
]{M} to choose a minimal prime $\fq$ of the ideal
$\fp\widehat{R}$ such that $\fq\cap R=\fp$. Then we notice that
the ring $\widehat{R}_{\fq}/\fp\widehat{R}_{\fq}$ is of dimension
0; so that it is Cohen--Macaulay. This observation in conjunction
with \cite[2.1.7 and 3.3.15]{BH} enables us to assume that $R$ is
complete. Hence, there are a Gorenstein local ring $S$ and a zero
height ideal $I$ of $S$ such that $R\cong S/I$. Therefore, $I\sim
Sx$ for some $x\in S$ by 2.2 . Now, in order to prove the
assertion, it is enough to show that $CM(R)\subseteq Gor(R)$. To
this end, let $\fp\in CM(R)$ and suppose $\fp'\in\Spec(S)$ is
such that $\fp'\supseteq I$ and that $\fp$ is the image of $\fp'$
under the natural homomorphism $S\longrightarrow R$. Then,
$R_{\fp}\cong S_{\fp'}/IS_{\fp'}$ possesses the canonical module
$\Hom_{S_{\fp'}}(R_{\fp},S_{\fp'})$ which is isomorphic to
$0:_{S_{\fp'}}IS_{\fp'}=(Sx)_{\fp'}$. Hence, by \cite[3.3.11]{BH},
$r(R_{\fp})=\mu((Sx)_{\fp'})=1$. Therefore, $R_{\fp}$ is a
Cohen--Macaulay ring of type 1; so that $R_{\fp}$ is Gorenstein.
\end{proof}

M. Hermann and N. V. Trung, in \cite{HT}, present a Buchsbaum
quasi-Gorenstein ring $(R,\fm,k)$ with $\dim R=3$ and
$H_{\fm}^{2}(R)\cong k\oplus k$ which is not Gorenstein. This
example motives us to the following theorem.

\begin{thm}
Let $(R,\fm)$ be a quasi-Gorenstein ring of dimension $n$ and
suppose that $H_{\fm}^{i}(R)=0$ for all integer $i$ with $n/2< i <
n$. Then $R$ is Gorenstein.
\end{thm}
\begin{proof}
If $n\leq 2$, then, by 2.1, $R$ is Cohen--Macaulay; hence it is
Gorenstein in view of 2.4. Thus, we may suppose that $n\geq 3$.
Since $\widehat{R}$ is faithfully flat over $R$, we have
$H_{\fm\widehat{R}}^{i}(\widehat{R})=0$ if and only if
$H_{\fm}^{i}(R)=0$. Therefore, in view of 2.3, we may further
assume that there are a Gorenstein local ring $(S,\fn)$ and a
zero height ideal $I$ of $S$ such that $R\cong S/I$. By 2.2, $I$
is linked to an ideal $Sx$ for some $x\in S$. Therefore, the
assumption $H_{\fn}^{i}(S/I)=0$ for all $n-(n-[\frac{n}{2}])< i <
n$ in conjunction with \cite[4.1]{Sch2} implies that $S/Sx$
satisfies Serre's condition $S_{n-[\frac{n}{2}]}$. Now, the
application of local cohomology with respect to $\fn$ to the
exact sequence $0\longrightarrow S/I\stackrel{x}{\longrightarrow}
S\longrightarrow S/Sx\longrightarrow 0$ shows that
$H_{\fn}^{i}(S/I)=0$ for all $i < n-[\frac{n}{2}]+1$. Therefore,
$H_{\fn}^{i}(S/I)=0$ for all $i < n$. Hence $R$ is
Cohen--Macaulay; so that, in view of 2.4, it is Gorenstein.
\end{proof}

\begin{rem}
The conclusion of the above theorem fails if one of the vanishing
conditions $H_{\fm}^{i}(R)=0$ $(n/2< i < n)$ is dropped. For
example, by \cite[2.2]{HT}, for given integers $n\geq 5$ and
$[\frac{n}{2}]< j < n$, there exists a non-Gorenstein
quasi-Gorenstein ring $R$ of dimension $n$ such that
$H_{\fm}^{i}(R)=0$ for all integers $i$ with $n/2< i < n$ and
$i\neq j$.
\end{rem}

An ideal $I$ in a local Cohen--Macaulay ring is called an almost
complete intersection ideal if $\mu (I)=\Ht I+1$. Evans and
Griffith, in \cite[2.1]{EG}, proved, over a certain regular local
ring, that an unmixed almost complete intersection ideal of
height two is  Cohen--Macaulay. In this manner we prove the
following.

\begin{thm}
Let $S$ be a local ring and suppose that $I$ is an ideal of $S$
such that $S/I$ is quasi-Gorenstein and that one of the following
two conditions holds.
\begin{enumerate}
  \item[(i)]$S$ is Gorenstein  and $I$ is almost complete
  intersection.
  \item[(ii)]$S$ is regular containing a field and $I$ is of height
  two.
\end{enumerate}
 Then
$S/I$ is Gorenstein.
\end{thm}
\begin{proof}
Assume (i) holds. Considering a maximal $S$-sequence which
constitutes a part of a minimal generating set of $I$ and passing
to the residue class ring, we can assume that $I=Sy$ for some
$y\in S$ and that the height of $I$ is zero. Then, by 2.2, $I\sim
Sx$ for some $x\in S$. Therefore, we have an exact sequence
$0\longrightarrow S/I\stackrel{x}{\longrightarrow} S
\stackrel{y}{\longrightarrow}
S\stackrel{x}{\longrightarrow}S\longrightarrow\cdots$, which
implies that $S/I$ is an $i$-th module of syzygy for all $i\geq
0$. Thus $S/I$ is Cohen--Macaulay; so that it is Gorenstein by 2.4.\\
Now, assume (ii) holds. One may consider a maximal $S$-sequence
in $I$, pass to the residue class ring and use 2.2 to construct
an ideal $J$ of height 2 which is generated by three elements such
that $I\sim J$. It now follows from \cite[2.1]{EG} that $S/J$ is
Cohen--Macaulay. Hence, by \cite[2.1]{PS}, $S/I$ is
Cohen--Macaulay. Therefore, by 2.4, $S/I$ is Gorenstein.
\end{proof}

\begin{rem}
In \cite[1.1]{K}, Kunz, essentially, proved that an almost
complete intersection ring is not quasi-Gorenstein. Then, as a
corollary, he pointed out that an almost complete intersection
ring is not Gorenstein. Conversely, one can employ 2.7(i), in the
case where $S$ is regular, to show that the above mentioned
corollary would imply \cite[1.1]{K}.
\end{rem}

As an application of the syzygy theorem, Evans and Griffith, in
\cite[2.2]{EG}, proved, over a regular local ring $R$ which
contains a field, that the residue class ring of a height two
prime ideal $\fp$ is complete intersection whenever
$\Ext^{2}_{R}(R/\fp,R)$ is principal. (Note that this last
condition is sufficient for $R/\fp$ to be quasi-Gorenstein.) The
following corollary is a slight generalization of this result.

The result \cite[Proposition 5]{S} of Serre, which indicates that
a height two Gorenstein ideal of a regular local ring is complete
intersection, in conjunction with 2.7(ii) leads immediately us to
the following corollary.

\begin{cor}
Let $S$ be a regular local ring containing a field and suppose
that $I$ is an ideal of height two such that $S/I$ is
quasi-Gorenstein. Then $I$ is complete intersection.
\end{cor}



\subsection*{Acknowledgment}
The authors would like to thank Professor M. Chardin for helpful
suggestion about Theorem 2.5. Also, the second author would like
to thank Professor D. Eisenbud for kind consideration to his
question about unmixed almost
complete intersection ideals.

\begin{thebibliography}{1}
\bibitem{AG} Y. Aoyama,  S. Goto, \textit{On the endomorphism ring of the canonical module},
J. Math. Kyoto Univ. \textbf{25}(1985) 21--30.

\bibitem{BSh} M. Brodmann,  R.Y. Sharp, \textit{Local Cohomology, An Algebric Introduction with Geometric
Application}, Cambridge University Press, Cambridge, 1998.

\bibitem{BH} W. Bruns,  J. Herzog, \textit{Cohen--Macaulay Rings}, revised version, Cambridge University
Press, Cambridge, 1998.

\bibitem{BE} D. A. Buchsbaum,  D. Eisenbud, \textit{Algebra structures for resolutions and some
 structure theorems for ideals of codimension 3},
Amer. J. Math. \textbf{99}(1977) 447--485.


\bibitem{EG} E. G. Evans,  P. Griffith, \textit{The syzygy problem},
Ann. of Math. \textbf{114}(1981) 323--333.

\bibitem{HT} M. Hermann,  N. V. Trung, \textit{Examples of Buchsbaum quasi-Gorenstein rings},
Proc. Amer. Math. Soc. V. 117, No. 3 (1993), 619--625.

\bibitem{K} E. Kunz, \textit{Almost complete intersections are not Gorenstein},
J. Alg. \textbf{28}(1974), 111--115.

\bibitem{M} H. Matsumura, \textit{Commutative ring theory},
Cambridge University Press, Cambridge, 1986.

\bibitem{PS} C. Peskine,  L. Szpiro, \textit{Liaison des varieties algebriques}, Invent. Math, \textbf{26}(1974),
271--302.

\bibitem{Sch1} P. Schenzel, \textit{A note on almost complete intersections},
Seminar Eisenbud-Singh-Vogel. Vol 2, Teubner-Texte Math., vol.
48, Teubner, Leipzig, 1982, pp. 49--54.

\bibitem{Sch2} P. Schenzel, \textit{Notes on Liaison and Duality},
J. Math. Kyoto Univ. \textbf{22}(1982), 485--498.

\bibitem{S} J. P. Serre, \textit{Sur les modules projectifs},
dans S\'{e}minaire Dubreil-Pisot (Alg\`{e}bre et th\'{e}orie des
nombres) 1960/1961, n$^{\circ}$2.

\end{thebibliography}
\end{document}